\documentclass[12pt]{amsart}
\usepackage{amsfonts}

\newtheorem{theorem}{Theorem}[section]

\newtheorem{lemma}[theorem]{Lemma}

\newtheorem{example}[theorem]{Example}

\newtheorem{remark}[theorem]{Remark}

\def\NN{\hbox{\sf I\kern-.13em\hbox{N}}}
\def\RR{\hbox{\sf I\kern-.14em\hbox{R}}}

\def\Cc{\hbox{\sf C\kern -.47em {\raise .48ex \hbox{$\scriptscriptstyle |$}}
   \kern-.5em {\raise .48ex \hbox{$\scriptscriptstyle |$}} }}

\newcommand{\be}{\begin{equation}}
\newcommand{\ee}{\end{equation}}

\newcommand{\cM}{{\mathcal M}}
\begin{document}

\baselineskip 5.8mm

\title[Uniform boundedness principle for non-linear operators on cones of functions]
{Uniform boundedness principle for non-linear operators on cones of functions}

\author{Aljo\v{s}a Peperko}
\date{\thanks{}\today}

\begin{abstract} 
\baselineskip 7mm
We prove an  uniform boundedness principle for the Lipschitz seminorm of continuous, monotone, positively homogeneous and subadditive mappings  on suitable cones of functions. The result is applicable to several classes of classically non-linear operators.
\end{abstract}

\maketitle

\noindent
{\it Math. Subj.  Classification (2010)}:  47H07, 47H05,  47H30, 47B65.
 \\

\section{Introduction and preliminaries}

Uniform boundedness principle for bounded linear operators (Banach-Steinhaus theorem) in one of the cornerstones of classical functional analysis (see e.g \cite{C90}, \cite{AA02}, \cite{EN00} and the references cited there). In this article we prove a new uniform boundedness principle for monotone, positively homogeneous, subadditive and Lipschitz mappings defined on a suitable cone of functions (Theorem \ref{uniform}). This result is applicable to several classes of classically non-linear operators (Examples \ref{MNuss03}, \ref{bounded_sup},  and Remarks  \ref{PF}, \ref{Hammerstein}).

Let $ \Omega$ be a non-empty set. 
Throughout the article let $X$ denote a vector space of all functions  $\varphi: \Omega \to \mathbb{R}$  or  
a vector space of all equivalence classes of (almost everywhere equal)
real measurable functions on $\Omega$, if $(\Omega, \cM, \mu)$ is a measure space. As usually,  $|\varphi|$ denotes the absolute value of $ \varphi \in X$.

 Let $Y \subset X$ be a vector space and let $Y_+$ denote the positive cone  of $Y$, i.e., the set of all $\varphi \in Y$ such that $\varphi (\omega) \ge 0$ for all (almost all) $\omega \in \Omega$. The space $Y$ is called an ordered vector  space with the partial ordering induced by the cone $Y_+$.  If, in addition, Y is a normed space it is called an ordered normed space. 
 The vector space $Y \subset X$ is called a vector lattice (or a  Riesz space) if for every  $\varphi,\psi \in Y$ we have a supremum and  infimum (greatest lower bound) in $Y$.   If, in addition, Y is a normed space and if $|\varphi| \le |\psi|$ implies $\|\varphi\| \le \|\psi\|$, then $Y$ is a called  normed vector lattice (or a normed  Riesz space). Note that in a normed vector lattice $Y$ we have 
$\||\varphi|\|= \|\varphi\|$ for all $\varphi \in Y$.  A complete normed vector lattice is called a Banach lattice. Observe that $X$ itself is a vector lattice.


 Let $Y \subset X$ be a normed space.  
The cone $Y_+$ is called normal if and only if there exists a constant $C>0$ such
that $\|\varphi\| \le C \|\psi\|$ whenever $ \varphi \le \psi$, $\varphi,\psi \in Y_+$. A cone $Y_+$ is normal if and only if there exists an equivalent monotone norm $||| \cdot|||$ on $Y$, i.e.,  
$|||\varphi||| \le |||\psi|||$ whenever $0 \le \varphi \le \psi$ (see e.g. \cite[Theorem 2.38]{AT07}).
A positive cone of a normed vector lattice  is  closed and normal.
Every closed cone in a finite dimensional
Banach space  is necessarily normal.  

Let $Z \subset Y$ be a cone (not necessarily equal to $Y_+$).
A cone $Z$ is said to be complete if it is a complete metric space in the topology induced by $Y$. In the case when $Y$ is a Banach space this is equivalent to $Z$ being closed in $Y$.  

  A mapping $A: Z \to Z$ is called positively homogeneous
(of degree 1) if $A(t\varphi) = tA(\varphi)$ for all $t \ge 0$ and $\varphi \in Z$. 
A mapping $A: Z \to Z$ is called  Lipschitz if there exists $L >0$ such that $\|A\varphi-A\psi \|\le L\|\varphi-\psi\| $ for all $\varphi,\psi \in Z$ and we denote $$\|A\|_{LIP} = \sup _{\varphi,\psi \in Z, \varphi \neq \psi} \frac{\|A\varphi-A\psi\|}{\|\varphi-\psi\|}.$$ If $A$ is Lipschitz and positively homogeneous, then   $$\|A\|_{LIP}= \sup _{\varphi,\psi, \in Z , \|\varphi-\psi\|=1}\|A\varphi-A\psi\| .$$
Note also that a  Lipschitz and positively homogeneous mapping $A$ on $Z$ is always bounded on $Z$, i.e.,
$$\|A\| = \sup _{\varphi \in Z, \varphi \neq 0} \frac{\|A\varphi\|}{\|\varphi\|}= \sup _{\varphi \in Z , \|\varphi\|=1}\|A\varphi\| $$
is finite and it holds $\|A\| \le \|A\|_{LIP}$. Moreover, a positively homogeneous mapping \\
 $A:Z \to Z$, which is continuous at $0$ is bounded  on $Z$ .

A set $K\subset Y$ is called a wedge if $K+K \subset K$ and if $t K \subset K$ for all $t\ge 0$. A wedge $K$ induces on $Y$ a vector preordering $\le _K$ ($\varphi \le _K \psi$ if and only if $\psi - \varphi \in K$), which is reflexive, transitive, but not necessary antisymmetric.

If $K \subset Y$ is a wedge, then 
$A: K \to Y$  is called subadditive  if 
$A(\varphi + \psi) \le  A \varphi + A \psi$ for $\varphi, \psi \in K$, and is called monotone (order preserving)   if $A\varphi \le A\psi$ whenever $\varphi \le \psi $, $\varphi,\psi \in K$. Note that in this definition of subadditivity and monotonicity
we consider on $Y$ (and on $K$)  
a partial ordering $\le _{Y_+}$ induced by $Y_+$ (not a preordering  $\le _K$). One of the reasons for this choice is that, for example,  it may happen that a non-linear map is   monotone with respect  to the ordering 
$\le _{Y_+}$, but it is not monotone with respect  to the preordering $\le _K$ (see, for instance, \cite[Section 5]{MN10} 
and max-type operators, or  \cite{LN08} and the "renormalization operators" which occur in discussing diffusion on fractals).  Moreover, for similar reasons wherever in our article we consider a subcone $Z \subset Y_+$  we consider on $Z$ a partial ordering $\le _{Y_+}$ induced by $Y_+$ (not a partial ordering $\le _Z$). Observe that in this setting the set $Z-Z$  is a vector subspace in 
$Y$ and thus a wedge.

In our main result (Theorem \ref{uniform}) we will consider a normed space $Y\subset X$ with a normal cone $Y_+$ and a complete subcone $Z \subset Y_+$ that satisfies $|\varphi - \psi| \in Z$ for all $\varphi, \psi \in Z$ and such that
$\|\varphi\|= \||\varphi|\|$ for all $\varphi = \varphi _1 - \varphi _2$ where $\varphi _1, \varphi _2  \in Z$. Since $X$ itself is a vector lattice the above assumptions make sense.
 Note also that a positive cone $Z=Y_+$ of each Banach lattice $Y$ or, in particular, of each Banach function space (see e.g.  \cite{AA02}, \cite{AB85}, \cite{Za83},  
\cite{P17}, \cite{DP16}, \cite{P17b} and the references cited there) satisfies these properties. For the theory of cones, wedges, linear and non-linear operators on cones and wedges, Banach ordered spaces,  Banach function spaces, vector and Banach lattices and applications e.g. in financial mathematics we refer the reader 
to  \cite{AA02}, \cite{AT07}, \cite{APV04}, \cite{Za83}, 
 \cite{AB85}, \cite{W99}, \cite{ABB90}, \cite{LT96}, \cite{JM14}, \cite{MP17}, \cite{MP17b}, \cite{MN02}, \cite{MN10}, 
 and the references cited there.

\section{Results}
We will need the following lemma.
\begin{lemma} Let $Y\subset X$ be a vector space and let  $Z \subset Y_+$ be a  subcone such that $|\varphi - \psi| \in Z$ for all 
$\varphi, \psi \in Z$. 
If $A: Z-Z \to Y$ is a subadditive and monotone mapping, then 
\be
|A\varphi - A\psi| \le A|\varphi - \psi|
\label{basic}
\ee
for all $\varphi , \psi \in Z$.

If, in addition, $Y$ is a normed space such that $Y_+$ is normal and $Z \subset Y_+$ is a subcone  such that
$\|\varphi\|= \||\varphi|\|$ for all $\varphi = \varphi _1 - \varphi _2$ where $\varphi _1, \varphi _2  \in Z$, and if $AZ \subset Z$ and $A$ is bounded on $Z$, then $A$ is Lipschitz on $Z$. 
\label{Lipschitz}
\end{lemma}
\begin{proof} Let $\varphi, \psi \in Z$. Since $A: Z-Z \to Y$ is a subadditive, we have 
$$A\varphi=A(\varphi-\psi+\psi) \le A(\varphi-\psi) + A\psi.$$
It follows that $A\varphi-A\psi \le A(\varphi-\psi) \le A|\varphi-\psi|$, since $A$ is monotone and $\varphi-\psi \le |\varphi-\psi|$.
Similarly one obtains that   $A\psi-A\varphi \le  A|\varphi-\psi|$, which proves (\ref{basic}).

Assume that, in addition, $Y$ is a normed space such that $Y_+$  is normal (with a normality constant $C$)  and $Z \subset Y_+$ a subcone   such that
$\| \varphi _1 - \varphi _2\|= \|| \varphi _1 - \varphi _2|\|$ for all  $\varphi _1, \varphi _2  \in Z$, and that $AZ \subset Z$ and $A$ is bounded on $Z$. It follows from (\ref{basic}) that
$$\|A\varphi -A \psi\| = \|| A\varphi - A\psi |\| \le C\| A|\varphi - \psi|\| \le C \|A\| \|\varphi - \psi\|$$
and thus $A$ is Lipschitz on $Z$ (and $\|A\|_{LIP} \le C \|A\|$), which completes the proof. 

\end{proof}

The following uniform boundedness principle is the central result of this article.

\begin{theorem} 
Let $Y\subset X$ be a normed space such that $Y_+$ is normal and let  $Z \subset Y_+$ be a complete subcone, such that $|\varphi - \psi| \in Z$ for all $\varphi, \psi \in Z$ and such that
$\|\varphi\|= \||\varphi|\|$ for all $\varphi \in Z-Z$.
Assume that $\mathcal{A}$ is a set of subadditive and monotone mappings $A: Z-Z \to Y$ such that $AZ \subset Z$ and that  each $A \in \mathcal{A}$ is 
positively homogeneous and  continuous on $Z$. 

If the set $\{A\varphi : A\in \mathcal{A} \}$ is bounded for each $\varphi \in Z$ (i.e.,  for each $\varphi \in Z$ there exists $M_{\varphi} >0$ such that $\|A\varphi\| \le M_{\varphi} $ for all $A\in \mathcal{A}$), then there exists $M>0$ such that $\|A\|_{LIP} \le M$
for all $A \in \mathcal{A}$.
\label{uniform}
\end{theorem}

\begin{proof} Since $Z$ is closed and each $A \in \mathcal{A}$ is continuous on $Z$ the set 
$$A_n=\{\varphi \in Z: \|A\varphi\|\le n \;\;\mathrm{for} \;\; \mathrm{all} \;\; A\in \mathcal{A}\}$$
is closed in $Y$ for each $n \in \mathbb{N}$. Moreover, $Z$ is a complete metric space and $Z= \cup _{n=1} ^{\infty} A_n$. By Baire's theorem there 
exist $n_0 \in \mathbb{N}$, $\varphi _0 \in Z$ and $\varepsilon >0$ such that  
an open ball $\mathcal{O}(\varphi _0, 3\varepsilon C)= \{\varphi \in Z: \|\varphi -\varphi _0\|< 3 \varepsilon C\}\subset A_{n_0},$
where $C$ is the normality constant of $Y_+$.

Let $\varphi,\psi\in Z$ such that $\|\varphi-\psi\|=1$ and $A\in \mathcal{A}$. 
Since $Z$ is a normal cone and $A$ is positively homogeneous on $Z$, we have by (\ref{basic})
\be
\|A\varphi-A\psi\|= \| |A\varphi-A\psi |\| \le C\|A|\varphi -\psi|\|= \frac{C}{\varepsilon}\|A(\varepsilon|\varphi -\psi|)\|.
\label{estimate}
\ee
Since $A$ is subadditive and monotone on $Z-Z$ we have
$$A(\varepsilon|\varphi -\psi|) = A(\varphi _0 + \varepsilon|\varphi -\psi| - \varphi _0) \le A\varphi _0 + A( \varepsilon|\varphi -\psi| - \varphi _0 ) \le A\varphi _0 + A( |\varepsilon|\varphi -\psi| - \varphi _0 |), $$
which together with (\ref{estimate}) implies
$$ \|A\varphi-A\psi\| \le \frac{C^2}{\varepsilon}\left( \|A\varphi _0 \|+ \|A|\varepsilon|\varphi -\psi|- \varphi _0|\| \right)\le  \frac{C^2}{\varepsilon}\left( n_0+ \|A|\varepsilon|\varphi -\psi|- \varphi _0|\| \right).$$
We also have 
\begin{equation}
||\varepsilon|\varphi -\psi|- \varphi _0|- \varphi _0 | < 3 \varepsilon |\varphi -\psi|. 
\label{abs}
\end{equation}

Indeed, if 
$\varepsilon|\varphi (\omega) -\psi (\omega)|- \varphi _0 (\omega) \le 0$, then 
$$||\varepsilon|\varphi(\omega) -\psi (\omega)|- \varphi _0(\omega)|- \varphi _0 (\omega)|= \varepsilon|\varphi(\omega) -\psi(\omega)| $$ 
and if $\varepsilon|\varphi (\omega) -\psi (\omega)|- \varphi _0 (\omega) > 0$, then
$$||\varepsilon|\varphi(\omega) -\psi (\omega)|- \varphi _0(\omega)|- \varphi _0 (\omega)|= |\varepsilon|\varphi(\omega) -\psi(\omega)|- 2\varphi _0 (\omega) | $$
$$\le \varepsilon|\varphi(\omega) -\psi(\omega)| + 2 \varphi _0 (\omega) <  3\varepsilon|\varphi(\omega) -\psi(\omega)|,$$
which proves (\ref{abs}). 

It follows from (\ref{abs}) that $\||\varepsilon|\varphi -\psi|- \varphi _0|- \varphi _0 \| \le 3C \varepsilon \|\varphi -\psi\|= 3C \varepsilon$ and thus \\
$|\varepsilon|\varphi -\psi|- \varphi _0| \in A_{n_0}$ and so 
$\|A|\varepsilon|\varphi -\psi|- \varphi _0|\|\le n_0$. Therefore
$$\|A\varphi-A\psi\| \le \frac{2C^2 n _0}{\varepsilon} $$
and so $\|A\|_{LIP} \le \frac{2C^2 n _0}{\varepsilon}$.
\end{proof}
\begin{remark} {\rm (i) Each $A\in \mathcal{A}$ satisfies $\|A\| \le \|A\|_{LIP} \le C\|A\|$ (see the proof of Lemma \ref{Lipschitz}). Therefore we could alternatively set $\psi =0$
in the proof above and prove an uniform upper bound 
$\|A\|  \le \frac{2C n _0}{\varepsilon} $ for all $A\in \mathcal{A}$, which gives the same conclusion.

(ii) In the proofs of Lemma \ref{Lipschitz} and Theorem \ref{uniform} we did not need the assumption $Z \cap (-Z) =\{0\}$, so it suffices to assume that $Z$ is a wedge in this two results (not necessarily a cone). 

(iii) Also the assumption on normality of $Y_+$ can be slightly weakened in Lemma \ref{Lipschitz} and Theorem \ref{uniform}. Instead of normality of $Y_+$ it suffices to assume that  there exists a constant $C>0$ such
that $\|\varphi\| \le C \|\psi\|$ whenever $ \varphi \le  \psi$, $\varphi,\psi \in Z$ (where again $\varphi \le  \psi$ means 
$\varphi \le_{Y_+} \psi$)
}
\end{remark}

Our results can be applied to various classes of  non-linear operators. In particular, they apply to various  max-kernel operators (and their isomorphic versions) appearing in the literature (see e.g. \cite{MN02},  \cite{MN10}, \cite{KM97}, \cite{AGN} and the references cited there). We point out the following two related examples from \cite{MN02}, \cite{MN10}, \cite{MP17}, \cite{MP17b}.

\begin{example} 
\label{MNuss03}
{\rm  Given $a>0$, let $Y=C[0,a]$ be Banach lattice of continuous functions on $[0,a]$ equipped with $\|\cdot\|_{\infty}$ norm. Consider the following
max-type kernel operators \\
$A: C[0,a] \to C[0,a]$ of the form
$$(A(\varphi))(s)=\max_ {t\in [\alpha (s), \beta (s)]}{k(s,t)\varphi(t)},$$
where $\varphi \in C[0,a]$ and $\alpha, \beta:[0,a]\to[0,a]$ are given continuous functions satisfying $\alpha \le \beta$.
The kernel  $k: S \to [0, \infty)$ is a given non-negative continuous function, where $S$
 denotes the compact set
$$S=\{(s,t)\in [0,a]\times[0,a]: t \in [\alpha (s), \beta (s)]\}.$$
It is clear that for $Z=C_+ [0,a]$ it holds $AZ \subset Z$. 
The eigenproblem of these operators arises in the study of periodic solutions of a class of 
differential-delay equations
$$\varepsilon y^{\prime}(t)=g(y(t),y(t-\tau)), \quad \tau=\tau(y(t)),$$
with state-dependent delay (see e.g. \cite{MN02}). 

The mapping $A:Y \to Y$ is  subadditive and monotone and is a positively homogeneous and Lipschitz on $Z$. 
Moreover, $\|A\|_{LIP}=\|A\|=   \max _{(s_0, s_1) \in \mathcal{S}_1 } k (s_0, s_1)$, where \\
$\mathcal{S}_1 =\bigl\{(s_0, s_1): s_0 \in [0,a], s_1 \in [\alpha(s_0), \beta (s_0)]\bigr\}.$ Cleary, our Theorem \ref{uniform} applies to sets of such mappings.

Consequently, our Theorem  \ref{uniform} implies also to isomorphic max-plus mappings (see e.g. \cite{MN02} and the references cited there) and a Lipschitz seminorm with respect to a suitably induced metric. Note that a related result for uniform boundedness   (in fact contractivity) result for a Lipschitz seminorm of  semigroups of max-plus mappings was stated in    \cite{KFPS16},  \cite{AKFPS16}. However, observe that the Lipschitz seminorm there is defined with respect to a different metric than in our case.
}
\end{example}

We also point out the following related example from \cite{MP17} and \cite{MP17b}.

\begin{example} 
\label{bounded_sup}
{\rm
Let $M$ be a nonempty set and let $Y$ be the set of all bounded real functions on $M$.
With the norm $\|f\|_{\infty}=\sup\{|f(t)|:t\in M\}$ and natural operations, $Y$ is a Banach lattice.
Let $Z=Y_+$ 
and
let $k:M\times M\to [0,\infty)$ satisfy $\sup \bigl\{k(t,s):t,s\in M\bigr\}<\infty$.
Let $A:Y\to Y$ be defined by $(Af)(s)=\sup\{k(s,t)f(t):t\in M\}$. Then $A:Y \to Y$ is  subadditive and monotone mapping that satisfies $AZ\subset Z$ and is a positively homogeneous and Lipschitz on $Z$, therefore  Theorem \ref{uniform} applies to sets of such mappings. It also holds
 $\|A\|_{LIP}=\|A\|=\sup \bigl\{k(t,s):t,s\in M\bigr\}$. 
In particular, if $M$ is the set of all natural numbers $\NN$, our results apply to infinite bounded non-negative matrices $k=[k(i,j)]$ (i.e., $k(i,j) \ge 0$ for all $i,j \in \NN$ and    $\|k\|_{\infty}=\sup _{i,j \in \NN} k(i,j) < \infty$). In this case, $Y= l^{\infty}$ and 
$Z=l^{\infty}_+ $ and $\|A\|_{LIP}= \|A\| =\|k\|_{\infty}$.
}
\end{example}

\begin{remark}{\rm The special case of Example \ref{bounded_sup} when  
$M=\{1, \ldots , n\}$ for some $n\in \NN$ is well known and studied under the name max-algebra (an analogue of linear algebra). Together with  its isomorphic versions (max-plus algebra and min-plus algebra also known as tropical algebra) it provides 
an attractive way of describing a class of non-linear problems appearing for instance in manufacturing and transportation scheduling, information 
technology, discrete event-dynamic systems, combinatorial optimization, mathematical physics, DNA analysis, ... (see e.g.  \cite{B98}, \cite{Bu10}, \cite{MP15}, \cite{MP12}, \cite{GMW17}, 
\cite{PS05} and the references cited there).
}
\end{remark}

\begin{remark}{\rm Our results apply also to more general max-type operators studied in \cite[Section 5]{MN10}. The authors considered there finite sums of more general operators than in Example \ref{MNuss03} defined on a Banach space of continuous functions 
and  their restrictions to suitable closed cones. The assumptions of our results are satisfied also for these mappings and therefore also for a special case of cone-linear Perron-Frobenius operators studied there.  
\label{PF}
}
\end{remark}

\begin{remark}{ \rm Theorem \ref{uniform} applies to several classes of nonlinear integral operators (under suitable assumptions on the kernels and on the defining nonlinearities) including Hammerstein type operators 
(see e.g. \cite[Chapter 12, p.338 ]{APV04}, \cite{BB75}, \cite{S08}), Uryson type operators (see e.g  \cite[Chapter 12, p.339 ]{APV04}, \cite{I09}) and Hardy-Littlewood type operators (see e.g. \cite{BS13}, \cite{I14}, \cite{MaP12}, \cite{NLW16}). 
\label{Hammerstein}
}
\end{remark}

\vspace{1mm}

\vspace{0.2cm}

\vspace{3mm}

\baselineskip 5mm

{\it Acknowledgments.} 
The author acknowledges a partial support of  the Slovenian Research Agency (grants P1-0222 and J1-8133).  He also thanks the referee for several useful comments and remarks.\\

\vspace{2mm}

\noindent  Aljo\v{s}a Peperko \\
Faculty of Mechanical Engineering \\
University of Ljubljana \\
A\v{s}ker\v{c}eva 6\\
SI-1000 Ljubljana, Slovenia\\
{\it and} \\
Institute of Mathematics, Physics and Mechanics \\
Jadranska 19 \\
SI-1000 Ljubljana, Slovenia \\
e-mails : aljosa.peperko@fs.uni-lj.si; aljosa.peperko@fmf.uni-lj.si

\end{document}